\documentclass[a4paper]{article}
\usepackage[affil-it]{authblk}
\usepackage[left=1in, right=1in, top=1in, bottom=1in]{geometry}
\usepackage{times}
\usepackage{amssymb}
\usepackage{bbm}
\usepackage{tikz}
\usepackage{titling}
\usepackage{lipsum}
\usepackage{amsmath}
\usepackage{amsthm}
\usepackage[utf8]{inputenc}
\usepackage[english]{babel}
\newcommand{\bx}{{\mathbf{X}}}
\newcommand{\bt}{{\mathbf{t}}}

\newcommand{\pr}{{\mathbb{P}}}
\newcommand{\ex}{{\mathbb{E}}}

\newcommand{\bz}{{\mathbf{Z}}}
\newcommand{\tk}{{\mathbf{\theta}}}

\pretitle{\begin{center}\Huge\bfseries}
	\posttitle{\par\end{center}\vskip 0.5em}
\preauthor{\begin{center}\Large\ttfamily}
	\postauthor{\end{center}}
\predate{\par\large\centering}
\postdate{\par}
\theoremstyle{definition}
\newtheorem{definition}{Definition}[section]
\newtheorem{thm}{Theorem}[section]
\newtheorem{lemma}{Lemma}[section]
\newtheorem{remark}{Remark}
\title{A Proof of the Herschel-Maxwell Theorem Using the Strong Law of Large Numbers}
\author{Somabha Mukherjee
	\thanks{Electronic address: \texttt{somabha@wharton.upenn.edu, somabhamukherjee@gmail.com}}}
\affil{Department of Statistics, Wharton School, University of Pennsylvania}
\date{\today}
\usepackage{fancyhdr}
\begin{document}
	
	\maketitle
	\thispagestyle{empty}
\begin{tikzpicture}
\draw (0,0) -- (16,0);
\end{tikzpicture}	
	\begin{abstract}
		  In this article, we use the strong law of large numbers to give a proof of the Herschel-Maxwell theorem, which characterizes the normal distribution as the distribution of the components of a spherically symmetric random vector, provided they are independent. We present shorter proofs under additional moment assumptions, and include a remark, which leads to another strikingly short proof of Maxwell's characterization using the central limit theorem.    
	\end{abstract}
	
\begin{tikzpicture}
\draw (0,0) -- (16,0);
\end{tikzpicture}

\tableofcontents

	\section{Introduction}
	The Herschel-Maxwell theorem is one of the many beautiful characterizations of the normal disribution. It states that if the distribution of a random vector with independent components is invariant under rotations, then the components must be identically distributed as a normal distribution. 
	
	As mentioned in \cite{charakt}, J.C. Maxwell addressed the following question: \textit{What is the distribution of velocities of the gas particles?} The argument behind Maxwell's claim that velocities are normally distributed, hinged upon two very natural assumptions about the distribution function, independence and rotation invariance. Even before Maxwell, astronomer J.F.W. Herschel addressed a similar issue while characterizing the errors in astronomical measurements. He assumed that the components of the two-dimensional errors in measurement are independent, and that the distribution of the error is independent of its direction.
	
	In this paper, we give a proof of the Herschel-Maxwell theorem using the strong law of large numbers, and give a remark about another unbelievably short proof of the theorem using the central limit theorem. The main tools of our analysis are characteristic functions and Haar's Theorem for rotation-invariant measures on the surface of the unit sphere in Euclidean spaces.       
	
	\section{Some Basic Properties of a Spherically Symmetric Distribution}\label{basic}
	\theoremstyle{definition}
	\begin{definition}\label{spherical}
		A random vector $\mathbf{X}$ taking values in $\mathbb{R}^n$ is said to have a spherically symmetric distribution, if $\mathbf{X}$ and $\mathbf{H} \mathbf{X}$ have the same distribution for every $n \times n$ real, orthogonal matrix $\mathbf{H}$.
		\end{definition}
In the following two theorems, we state some basic properties of a spherically symmetric distribution. 
\begin{thm}\label{one}
The entries of a spherically symmetric random vector have the same distribution. Moreover, if that distribution has a finite mean, then the mean must be $0$, and if that distribution has finite second moment, then any two distinct entries of the random vector are uncorrelated. 
\end{thm}
\begin{thm}\label{two}
The random vector $\bx = (X_1,...,X_n)^{\textrm{T}}$ has a spherically symmetric distribution if and only if its characteristic function $\phi$ satisfies $\phi (\bt) = \ex e^{i||\bt||X_1}$ for all $\bt \in \mathbb{R}^n$.  
\end{thm}
It follows immediately from Theorem \ref{two}, that if a random vector $(X_1,...,X_n)^{\textrm{T}}$ follows a spherically symmetric distribution, then so does all its subvectors. We now state and prove some sort of a ``converse" of this fact, under an additional assumption,  which will play a crucial role in our main proof. 
\begin{thm}\label{three}
Let $F$ be a distribution on $\mathbb{R}$ with the property that if $X_1$ and $X_2$ are independent observations from $F$, then $(X_1,X_2)^{\textrm{T}}$ has a spherically symmetric distribution. Then for every $n$, if $X_1,...,X_n$ are independent observations from $F$, $(X_1,...,X_n)^{\textrm{T}}$ has a spherically symmetric distribution.
\end{thm}
\begin{proof}
In view of Theorem \ref{two}, it suffices to show that $\ex e^{i \sum_{j=1}^n t_j X_j} = \ex e^{i \left(\sqrt{\sum_{j=1}^n t_j^2}\right) X_1}$ for all $n$ and for all $t_1,...,t_n$. By Theorem \ref{two}, this is true for $n = 1$ and $2$. Assume that the proposition holds for some $n$. Let $X_1,...,X_n,X_{n+1}$ be $(n+1)$ independent observations from $F$. Then, by our induction hypothesis, we have: $$\ex e^{i \sum_{j=1}^{n+1} t_j X_j} = \left(\ex e^{i \sum_{j=1}^n t_j X_j}\right)\left(\ex e^{i t_{n+1} X_{n+1}}\right) = \left(\ex e^{i \left(\sqrt{\sum_{j=1}^n t_j^2}\right) X_1}\right)\left(\ex e^{i t_{n+1} X_{n+1}}\right) = \ex e^{i \left(\sqrt{\sum_{j=1}^{n+1} t_j^2}\right) X_1}$$ for all $t_1,...,t_{n+1}$. We are done.
\end{proof}
\begin{thm}\label{four}
Let $X$ and $Y$ be two independent random variables. Suppose that $(X,Y)^{\textrm{T}}$ has a spherically symmetric distribution. Then, $\pr(X=0)$ is either $0$ or $1$.
\end{thm}
\begin{proof}
Suppose, towards a contradiction, that $0<\pr(X=0)<1$. By Theorem \ref{one}, $X$ and $Y$ have the same distribution. Since $X \stackrel{\textrm{d}}{=} X \cos \tk + Y \sin \tk$ for every $\tk \in \mathbb{R}$, we have for all $\tk \in \mathbb{R}$ : 
  \begin{eqnarray*}
 &&\pr\left(X \cos \tk + Y \sin \tk = 0 , (X,Y) \neq (0,0)\right)\\ &=& \pr(X \cos \tk + Y \sin \tk = 0) - \pr(X=0,Y=0)\\ &=& \pr(X=0) - \pr(X=0)\pr(Y=0)\\ &=& \pr(X=0) \left(1-\pr(X=0)\right) > 0.
  \end{eqnarray*} 
However, the sets $\big\{(x,y) \neq (0,0)  : x \cos \tk + y \sin \tk = 0\big\}~\left(0 \leq \tk \leq \frac{\pi}{2}\right)$  are pairwise disjoint, and they form an uncountable collection. This contradicts the fact that for any set in this collection, the probability of $(X,Y)$ belonging to that set is positive.  
\end{proof}

\section{The Spherical Symmetry Characterization and its Proof}\label{proof}
We will require a simple version of Haar's Theorem for rotation-invariant measures on $\mathcal{S}^{n-1}$, the surface of the unit sphere in $\mathbb{R}^n$. It is stated below.
\begin{thm}\label{five}
Let $\mu$ be a rotation-invariant Borel probability measure on $\mathcal{S}^{n-1}$ i.e. $\mu(\mathbf{H} B) = \mu(B)$ for every Borel set $B \subseteq \mathcal{S}^{n-1}$ and every $n \times n$  orthogonal matrix $\mathbf{H}$. Then, $\mu$ is the uniform measure on $\mathcal{S}^{n-1}$. 
\end{thm}
It follows from Theorem \ref{five}, that if $\bx$ is a unit norm, spherically symmetric random vector in $\mathbb{R}^n$, then $\bx$ has the uniform distribution on $\mathcal{S}^{n-1}$. We are now ready to state and prove the main result of this paper. 
\begin{thm}\label{six}
Let $X$ and $Y$ be two independent random variables. Suppose that $(X,Y)^{\textrm{T}}$ has a spherically symmetric distribution.  Then, $X$ and $Y$ are identically distributed as a normal distribution with mean $0$ (and possibly $0$ variance). 
\end{thm}
\begin{proof}
By Theorem \ref{one}, $X$ and $Y$ have the same distribution, say $F$. By Theorem \ref{four}, $\pr(X=0)$ is either $0$ or $1$. In the latter case, $X$ has the normal distribution with mean $0$ and variance $0$. So, assume that $\pr(X=0) = 0$.

Generate a sequence $\{X_n\}_{n=1}^\infty$ of independent random variables from the distribution $F$, and a sequence $\{Z_n\}_{n=1}^\infty$ of independent $N(0,1)$ random variables. For each $n$, call $\bx_n = (X_1,...,X_n)^{\textrm{T}}$ and $\bz_n = (Z_1,...,Z_n)^{\textrm{T}}$. It follows from Theorem \ref{three} that $\bx_n$ has a spherically symmetric distribution i.e. for every $n \times n$ orthogonal matrix $\mathbf{H}$, $\bx_n$ and $\mathbf{H} \bx_n$ have the same distribution. So, $$\frac{\bx_n}{||\bx_n||} \stackrel{\textrm{d}}{=} \frac{\mathbf{H} \bx_n}{||\mathbf{H} \bx_n||} = \mathbf{H} \frac{\bx_n}{||\bx_n||}$$ for every $n$ and every $n \times n$ orthogonal matrix $\mathbf{H}$. Thus, $\frac{\bx_n}{||\bx_n||}$ is a unit norm spherically symmetric random vector in $\mathbb{R}^n$, and hence, follows the uniform distribution on $\mathcal{S}^{n-1}$. By the same argument,  $\frac{\bz_n}{||\bz_n||}$ also follows the uniform distribution on $\mathcal{S}^{n-1}$. Hence, $\frac{\bx_n}{||\bx_n||} \stackrel{\textrm{d}}{=} \frac{\bz_n}{||\bz_n||}$ for all $n$. This, in turn, implies that $\frac{\sqrt{n} X_1}{||\bx_n||} \stackrel{\textrm{d}}{=} \frac{\sqrt{n} Z_1}{||\bz_n||}$ for all $n$. By the Strong Law of Large Numbers, the right hand side converges almost surely to $Z_1$. Observe that $\ex X_1^2 < \infty$, since otherwise, by the Strong Law of Large Numbers for independent and identically distributed random variables with expectation $+\infty$, it would follow that the left hand side converges almost surely to $0$, a contradiction. So, by the Strong Law of Large Numbers for finite mean, the right hand side converges almost surely to $\frac{X_1}{\sqrt{\ex X_1^2}}$. Hence, $\frac{X_1}{\sqrt{\ex X_1^2}} \stackrel{\textbf{d}}{=} Z_1$ and we are done.
\end{proof}

\begin{remark}
	A slight modification of the proof of Theorem \ref{six} yields the following result:
	\begin{thm}\label{seven}
	Suppose that $\{X_n\}_{n=1}^\infty$ is a sequence of random variables satisfying the following conditions:\\
	1. $\pr(X_1 = 0) = 0$,\\
	2. $\ex X_1^4 < \infty$,\\
	3. $(X_1,...,X_n)^{\textrm{T}}$ is spherically symmetric for all $n \geq 1$,~ and\\
	4. $\textrm{Cov}(X_i^2,X_j^2) = 0$ for all $1 \leq i < j$.\\
	Then, $X_1,X_2,...$ are identically distributed as a normal distribution with mean $0$ and positive variance. 
	\end{thm}
	The only observation needed before replicating the proof of Theorem \ref{six} is that, under the above conditions, $\frac{||\bx_n||}{\sqrt{n}} \xrightarrow{\textrm{P}} \sqrt{\ex X_1^2}$. Theorem \ref{seven} is probably interesting only from the angle that the independence of the $X_n$'s can be relaxed in lieu of some additional assumptions, in order to arrive at the same normal characterization.    
\end{remark}

\section{Shorter Proofs Under Additional Moment Assumptions}\label{twoline}
Theorem \ref{six} has shorter proofs under additional assumptions of finiteness of the first and second moments of $X$. Suppose that we only have $\ex |X| < \infty$. Since $X$ has a symmetric distribution around $0$, this condition is equivalent to the existence of $\ex X$. In this case, the characteristic function $\phi$ of $X$ is differentiable on $\mathbb{R}$.\\\\
By an application of Theorem \ref{two}, we have $\phi(s) \phi(t) = \phi\left(\sqrt{s^2 + t^2}\right)$ for all $s , t \in \mathbb{R}$.  Since the distribution of $X$ is symmetric around $0$, $\phi$ is a real valued, even function. We claim that $\phi(t) > 0$ for all $t \in \mathbb{R}$. If not, then since $\phi(0) = 1$, by the intermediate value theorem, there is a $t_0 \in \mathbb{R}$, such that $\phi(t_0) = 0$. Since $\phi(t) = \left[\phi\left(\frac{t}{\sqrt{2}}\right)\right]^2$ for all $t \in \mathbb{R}$, an easy induction gives $\phi(t) = \left[\phi\left(\frac{t}{2^{\frac{n}{2}}}\right)\right]^{2^n}$ for all $t \in \mathbb{R}$ and all $n \geq 1$. This implies that $\phi\left(\frac{t_0}{2^{\frac{n}{2}}}\right) = 0$ for all $n \geq 1$, which is not possible, since $\phi$ is continuous at $0$ and $\phi(0) = 1$. This proves our claim.\\\\
If we denote $\log \phi$ by $\psi$, then we have $\psi(s) + \psi(t) = \psi \left(\sqrt{s^2 + t^2}\right)$ for all $s , t \in \mathbb{R}$. Taking partial derivative with respect to $s$ on both sides of the above identity, we get: $$\psi'(s) = \psi' \left(\sqrt{s^2 + t^2}\right) \left(\frac{s}{\sqrt{s^2 + t^2}}\right)\hspace{0.3cm}\textrm{for all}\hspace{0.15cm} (s,t) \neq (0,0).$$ This implies that there is a constant $c$ such that $\frac{\psi'(s)}{s} = c$ for all $s \neq 0$. Solving this differential equation and remembering that $\psi$ is continuous at $0$ with $\psi(0) = 0$, we get $\psi(s) = \frac{c s^2}{2}$ for all $s \in \mathbb{R}$. Since $\psi(s) \leq 0$ for all $s$, we must have $c \leq 0$. Now, $\phi(s) = e^{\frac{c s^2}{2}}$ for all $s \in \mathbb{R}$ implies that $X \sim N\left(0,-c\right)$ and we are done.\\

If further, we assume that $\ex X^2 < \infty$, the proof turns out to be surprisingly short, and is given below.
\begin{lemma}\label{lem}
Let $\{X_n\}_{n=1}^\infty$ be a sequence of independent and identically distributed random variables, satisfying that $(X_1,...,X_n)^{\textrm{T}}$ has a spherically symmetric distribution for all $n$. For each $n$, denote the partial sum $\sum_{i=1}^n X_i$ by $S_n$. Then, $\frac{S_n}{\sqrt{n}}~(n=1,2,...)$ are identically distributed as $X_1$.
\end{lemma}
\begin{proof}
For each $n$, let $\mathbf{H}_n$ denote the orthogonal matrix whose first row is $\left(\frac{1}{\sqrt{n}},\frac{1}{\sqrt{n}},...,\frac{1}{\sqrt{n}}\right)$ and let $\bx_n = (X_1,...,X_n)^{\textrm{T}}$. Since $\bx_n$ and $\mathbf{H}_n \bx_n$ have the same distribution, their first entries have the same distribution. 
\end{proof}
Now, consider proving Theorem \ref{six} under the assumption $\ex X^2 < \infty$. The case $\ex X^2 = 0$ is trivial, so assume that $\ex X^2 > 0$. As in the proof of Theorem \ref{six}, generate  a sequence $\{X_n\}_{n=1}^\infty$ of independent random variables from the common distribution of $X$ and $Y$, and for each $n$, let $S_n = \sum_{i=1}^n X_i$. By Theorem \ref{three} and Lemma \ref{lem}, $X_1 \stackrel{\textbf{d}}{=} \frac{S_n}{\sqrt{n}}$ for all $n$. By Theorem \ref{one}, $\ex X = 0$. Hence, by the Central Limit Theorem, $\frac{S_n}{\sqrt{n}} \xrightarrow{\textbf{d}} N(0,\ex X^2)$. So, $X_1 \sim  N(0,\ex X^2)$.\\
\begin{remark}
	If $\{X_n\}_{n=1}^\infty$ is an i.i.d. sequence of random variables with $S_n \stackrel{def}{=} \sum_{i=1}^n X_i$, and if $\frac{S_n}{\sqrt{n}}$ converges in distribution to a limit, then $\ex X_1^2 < \infty$ (see exercise $3.4.3$ of \cite{durrett}). The finiteness of $\ex X^2$ is now an immediate consequence of this fact and Lemma \ref{lem}, which in turn gives a second proof of Theorem \ref{six}.     
	\end{remark}
	
\section{Conclusion}\label{conclusion}
\hspace{0.5cm}Theorem \ref{six} appears in \cite{charakt} (Theorem $0.0.1$) and an early proof of it appears in \cite{bartlett}. The treatment in \cite{bartlett} is however not very rigorous on probabilistic grounds. Corollary 10 of \cite{cltsymmetry} is a consequence of Theorem \ref{six}.

Our first proof of Theorem \ref{six} can be divided into two broad ideas. The first idea is to derive the spherical symmetry property of any number of independent observations from a distribution based on the knowledge of the spherical symmetry of two independent observations from that distribution. The second idea is to use the outcome of the first idea along with the Strong Law of Large Numbers, to conclude the result. In the process, the unit norm spherical symmetry characterization of the uniform distribution on the surface of an $n$ dimensional sphere was crucially used. The main advantage of this proof is that it is free of any calculation trickery, and is purely conceptual.

It is possible to give a more ``direct" proof of Theorem \ref{six} by solving the functional equation $\phi(s) \phi(t) = \phi\left(\sqrt{s^2 + t^2}\right)$ for all $s , t \in \mathbb{R}$, for a general characteristic function $\phi$. However, this approach relies strongly on the independence assumption of the random variables, and cannot, for example, be used to prove Theorem \ref{seven}. 
 
\section*{Acknowledgement}
The author thanks Professor J. Michael Steele for one of his assignment problems in the STAT 930 course offered by the University of Pennsylvania, which was the source of the idea behind the proof of Theorem \ref{six}.  




\end{document}